\documentclass[options]{amsart}

\usepackage{latexsym,ifthen,amssymb}

\newtheorem{theorem}{Theorem}[section]

\newtheorem{lemma}[theorem]{Lemma}

\newtheorem{emphasise}[theorem]{ }
\newtheorem {question}[theorem]{Question}
\theoremstyle{definition}
\newtheorem{definition}[theorem]{Definition}

\theoremstyle{remark}
\newtheorem{remark}[theorem]{Remark}

\newtheorem{Ramsey's theorem}[theorem]{Ramsey's theorem}

\def\c1{{\cdot 1}}

\def\be{\begin{equation}}
\def\ee{\end{equation}}
\def\bea{\begin{eqnarray}}
\def\eea{\end{eqnarray}}
\def\bd{\begin{displaymath}}
\def\ed{\end{displaymath}}
\def\bda{\begin{eqnarray*}}
\def\eda{\end{eqnarray*}}
\def\bsm{\begin{small}}
\def\esm{\end{small}}

\title{Generalized low solution
of $\mathsf{RT}_k^1$ problem }

\date{}

\author{Lu Liu }
\email{g.jiayi.liu@gmail.com}
\address{Department of Mathematics and Statistics, Central South University,
Changsha, 410083, China}

\thanks{}


\begin{document}

\begin{abstract}
We study the "coding power" of an arbitrary
$\mathsf{RT}_k^1$-instance. We prove that every
$\mathsf{RT}_k^1$-instance admit non trivial
generalized low
solution. This is somewhat related to a problem proposed
by Patey. We also answer a question proposed by
Liu, i.e., we prove that there exists a
$\mathbf{0}'$-computable
$\mathsf{RT}_3^1$-instance, $I_3^1$, such that every
$\mathsf{RT}_2^1$-instance admit a non trivial
solution that does not compute any non trivial solution
of $I_3^1$.
\end{abstract}
\maketitle

\section{Introduction}

Reverse mathematics is a field that studies
the proof theoretic strength of mathematical
theorems.
Many theorems are surprisingly found
to be equivalent to one of the big five
axioms \cite{simpson1999subsystems}. Ramsey's theorem
for pairs is probably
one of the most famous exception.

Ramsey's theorem for single
integers, $\mathsf{RT}_k^1$, is not
interesting in the sense of
reverse mathematics. Because over $\mathsf{RCA}_0$
$\mathsf{RT}_k^1$ is trivial.
However, the "coding power" of an \emph{arbitrary}
instance of $\mathsf{RT}_k^1$ attracts more
and more attention since many admitting-homogenous-set
theorems induced by binary relations reduce to
the study of  $\mathsf{RT}_k^1$
instance. For example, $\mathsf{RT}_2^2, \mathsf{EM}$.
Here $\mathsf{EM}$ is the Erd\H{o}s-Moser theorem which says
that every infinite tournament contains
an infinite transitive subtournament
(see also \cite{lerman2013separating}).

In this paper, we prove two theorems about
$\mathsf{RT}_k^1$ instance. The first says
that every instance of $\mathsf{RT}_k^1$
admit generalized low solution. The second
theorem prove the existence of a
Muchnick degree of the solutions
of an instance of
$\mathsf{RT}_3^1$ that can not be
reduced to that of
any instance of $\mathsf{RT}_2^1$.
These results are of technical interest
and are related to some recent arising
questions in reverse mathematics
as illustrated in subsection \ref{RT21subsecmain}.
We begin by introducing
definitions of $\mathsf{RT}_k^n$ and
review of the related literature.

\begin{definition}
Let $[X]^k$ denote $\{F\subseteq X: |F|=k\}$.
A k-coloring on $X$
is a function $c: [X]^n\rightarrow \{1,2\ldots k\}$.
A
set $H\subseteq X$ is homogeneous for $c$
iff $c$ is constant on $[H]^k$.
A stable $2-$coloring is a function
$c: [X]^2\rightarrow \{1,2\}$
such that there exists $i\in\{1,2\}$
$|\{x\in X: c(x)\ne i\}|<\infty$.
\end{definition}

\begin{definition}[Ramsey's theorem \cite{ramsey1930problem}]
 ($\mathsf{RT}_k^n$) For any $n,k$, every k-coloring of $[\mathbb{N}]^n$
 admits an infinite homogeneous set.

 The stable Ramsey's theorem
 for pair, $\mathsf{SRT}_k^2$, is $\mathsf{RT}_2^2$
 restricted to stable colorings.
\end{definition}

 \begin{definition}[$\mathsf{COH}$]
Let $S_0,S_1,\cdots$ be a sequence of
sets. A set $C$
 is cohesive wrt $S_0,S_1,\cdots$ iff
 $(\forall i\in\omega) |C\cap S_i|<\infty\vee
 |C\cap \overline{S_i}|<\infty  $.

$(\mathsf{COH})$: For every
uniform sequence $S_0,S_1,\cdots$
there exists an infinite cohesive
set.

 \end{definition}
 For more
details see also \cite{Cholak2001} or
\cite{hirschfeldt2014slicing}.
There is a lot of literature
on Ramsey theorems for pairs.
To mention a few,
Simpson \cite{simpson1999subsystems}
and Jockusch \cite{jockusch1972ramsey}
proved that over
$\mathsf{RCA}_0$,
 $\mathsf{RT}_k^n$ is equivalent to
 $\mathsf{ACA}_0$ for $n>2$.
 Jockusch \cite{jockusch1972ramsey}
 also showed that $\mathsf{WKL}_0$ does
 not imply $\mathsf{RT}_2^2$.
 Seetapun and Slaman
 in their celebrated paper
 \cite{seetapun1995strength}
   proved that
 $\mathsf{RT}_2^2$ does not imply
 $\mathsf{ACA}_0$ over $\mathsf{RCA}_0$.
 Cholak, Jockusch and Slaman \cite{Cholak2001}
 proved that $\mathsf{RT}_2^2$ is equivalent
 to $\mathsf{SRT}_2^2+\mathsf{COH}$
 over $\mathsf{RCA}_0$. Their paper also
 create one of the most important technique
 based on Mathias forcing,
 the $low_2$ construction.
 Liu \cite{liu2012rt},\cite{liu2015cone}
 separate $\mathsf{RT}_2^2$ from
 $\mathsf{WKL}_0$ and $\mathsf{WWKL}$
 respectively. Therefore, combining
 with Jockusch  \cite{jockusch1972ramsey}
 and the fact $\mathsf{RT}_2^2$ is not
 equivalent to $\mathsf{RCA}_0$,
 $\mathsf{RT}_2^2$ is not equivalent to
 any of the "big five". Most recently,
 Chong, Slaman and Yang \cite{chong2014metamathematics}
 proved that $\mathsf{SRT}_2^2$
 does not imply $\mathsf{COH}$ and thus
  $\mathsf{RT}_2^2$. This settles a long standing problem.
  However, their model is nonstandard and thus leave
  the question that whether $\mathsf{SRT}_2^2$
  imply $\mathsf{COH}$ in standard arithmetic model.
 Another important progress is Patey and
 Yokoyama \cite{patey2016proof},
 they proved that over $\mathsf{RCA}_0$
 $\mathsf{RT}_2^2+\mathsf{WKL}_0$ is
 $\tilde{\Pi}_3^0$ conservative. There they
invent the notion of $\alpha-$large
 etc that have prospective applications
 to other problems.

\subsection{Main results}
\label{RT21subsecmain}

In this subsection we introduce
 our main results and
 the relation
about our main results with the recent
progress in reverse mathematics.

Patey \cite{patey2015open} section 2
proposed several questions concerning
computational complexity of
solutions to $\Delta_2^0$ instance of
$\mathsf{RT}_2^1$.
For example, "whether any $\Delta_2^0$
$\mathsf{RT}_2^1$ instance
admit a solution that is both $\Delta_2^0$
and $low_2$";
"whether there exists a $\Delta_2^0$
$\mathsf{RT}_2^1$ such that any $\Delta_2^0$ non trivial
solution of which is high".
These questions are  related to the
currently most concerned problem in reverse mathematics,
whether $\mathsf{SRT}_2^0$ implies
$\mathsf{COH}$ in $\omega$-model.
We here prove that any instance of $\mathsf{RT}_2^1$
admit a generalized low solution.

\begin{theorem}\label{RT21th1}
For every instance of $\mathsf{RT}_2^1$,
$I_2^1$, there exists a non trivial
solution to $I_2^1$, $G$, such that
$G' =_T G\oplus \mathbf{0}'$.

\end{theorem}

Liu \cite{liu2015cone} proposed the
question that whether there exists
an instance of $\mathsf{RT}_3^1$ such that
of which the
 solution set is not Muchnik reducible
to any solution set of any instance of
$\mathsf{RT}_2^1$. We here give a positive
answer in theorem \ref{RT21th2}.
Similar results have been obtained independently
by 
Dzhafarov
\cite{dzhafarov2015cohesive},
Hirschfeldt andd Jockusch
\cite{hirschfeldt2015notions},
Patey \cite{patey2015weakness}.

The question is of technical interest.
Patey \cite{patey2015strength} proved that there exists
a $\mathbf{0}'$-computable instance of
$\mathsf{TT}_2^1$
, $I_{\mathsf{TT}^2_1}$ such that the solution
set of which is not Muchnik reducible
to any instance of $\mathsf{RT}_2^1$.
The proof relies on complexity of solution
space of $\mathsf{TT}_2^1$, i.e., fix an
arbitrary long initial segment of some
instance of $\mathsf{TT}_2^1$, $\rho$,
there exists sufficiently large $M$ such that
  given
any finitely many $M-$long initial segments
of  $\mathsf{TT}_2^1$-solutions
of $\rho$, namely
$\tau_i,i=1,\cdots,n$,
there exists some extension of $\rho$,
namely $\gamma$, such that any solution
to any instance extending $\gamma$ avoids
$\tau_i, i\leq n$.
However, this easy-avoidance property
does not holds for $\mathsf{RT}_3^1$.
$\mathsf{RT}_3^1$ codes the solutions
in a much more compact fashion. This is
reflected by the fact that any instance of
$\mathsf{RT}_3^1$ computes a solution of
itself.
The proof of theorem \ref{RT21th2}
employs the method in
Liu \cite{liu2015cone}.

\begin{theorem}\label{RT21th2}
There exists a $\mathbf{0}'$-computable
 instance of $\mathsf{RT}_3^1$,
$I_3^1$, such that
for every instance of $\mathsf{RT}_2^1$,
$I_2^1$,
there exists a non trivial solution to
$I_2^1$, $G_2^1$ such that $G_2^1$ does not
 compute any non trivial solution to
 $I_3^1$.

\end{theorem}

Actually, by the proof of theorem \ref{RT21th2},
it is plain to see that,
\begin{theorem}

There exists a $\mathbf{0}'$-computable
 instance of $\mathsf{RT}_{k}^1$,
$I_k^1$, such that
for every instance of $\mathsf{RT}_{k-1}^1$,
$I_{k-1}^1$,
there exists a non trivial solution to
$I_{k-1}^1$, $G_{k-1}^1$ such that $G_{k-1}^1$ does not
 compute any non trivial solution to
 $I_k^1$. Where $k-1\geq 1$ is arbitrary.
\end{theorem}

The rest of this paper is devoted to the
proof of the two theorem. We also propose
new problems in section \ref{RT21secfurther}

\section{Notations}
We write $\rho *\tau$ to denote the string
that concatenate $\tau$ after $\rho$.
We sometimes regard a binary string
$\rho\in 2^{<\omega}$ as a set
$\{j:\rho(j)=1\}$, and
write $\rho\subseteq X$ for
set containing relation,
$\rho-\tau$ for set minus operation,
$\rho\cap X$ for set intersection
operation. For a sequence of string,
$\cdots
\rho^i\subset \rho^{i+1}\cdots, i\in\omega$,
we write
$\bigcup\limits_{i\in\omega}\rho^i$
for the string $X$ such that
$(\forall i\in\omega)
X\upharpoonright_{1}^{|\rho^i|}=\rho^i$.
When $\rho\cup\tau$ denote set union,
we make assertion.
We write $\rho\subset Y$ iff
$\rho$ is prefix of $Y$.
\emph{Empty string} is denoted by
$\varepsilon$.

For a tree $T$, $[T]$ is
the set of paths through $T$.

For a set $X$ we write
$X'$ for the canonical jump,
i.e.,
$X'(n) = 1$ iff $\Phi_n^X(n)\downarrow$.
We write $\mathbf{0}'$ for the jump
(Turing) degree.

\section{Proof of theorem \ref{RT21th1}}
\subsection{Forcing conditions}\label{RT21subsec-1}
Firstly, recall the Mathias forcing.
\begin{definition}
\label{def1}
\

\begin{itemize}
\item A Mathias condition is a pair $(\sigma,X)$
with $\sigma \in 2^{<\omega}$ and $X \in 2^\omega$.

\item $(\tau,Y)$ extends the Mathias condition $(\sigma,X)$
iff $\sigma \subset \tau$ and $Y/\tau \subseteq X/\sigma$.
Write $(\tau,Y)\leq (\sigma,X) $ to denote
the extension relation.

\item A set $G$ satisfies the Mathias condition
$(\sigma,X)$ if $\sigma \subset G$
 and $G \subseteq X/\sigma$.
\end{itemize}

\end{definition}

We say
string $X\in 2^\omega$ codes an
\emph{ordered $k$-partition} iff
\begin{itemize}
\item $X=X_0\oplus X_1\oplus\cdots\oplus X_{k-1}$ and

\item $\cup_{i=0}^{k-1}X_i=\omega$
\end{itemize}

 A class $P$ is a $k$-partition class iff
  $\forall X\in P$, $X$ codes an ordered $k$-partition.

  \begin{definition}
In the proof of theorem \ref{RT21th1},
a forcing condition is a tuple,
$((\rho_{1,l},\rho_{1,r})
,\ldots,(\rho_{k,l},\rho_{k,r}),P,k)$,
where $k>0$ indicates the
number of partitions
and $P$ is a $\Pi^{0}_1$ $k$-partition class.
\end{definition}

\begin{definition}
\label{def-sat}
We say set
$G$ \emph{satisfies}
condition $((\rho_{1,l},\rho_{1,r})
,\ldots,(\rho_{k,l},\rho_{k,r}),P,k)$
part $j$ left side
iff there exists $X_1 \oplus
\cdots \oplus X_{k} \in P$
such that
$G$ satisfies  $(\rho_{i,l},X_i)$.
Similarly for right side.

We say set $G$ \emph{satisfies}
condition $c$ on part $j$ iff
it satisfies condition $c$ part $j$
on left or right side.
\end{definition}

  \begin{definition}
\label{extend}
We say condition
$d=((\tau_{1,l},\tau_{1,r}),\ldots,(\tau_{m,l},\tau_{m,r}),Q,
m)$ \emph{extends} condition

$c=((\sigma_{1,l},\sigma_{1,r}),\ldots,(\sigma_{k,l},\sigma_{k,r}),
P,k)$,
denoted by $d\leq c$, iff there is a function $f : m
\rightarrow k$ satisfying
 ($\forall i\leq m\ \forall Y_1 \oplus
\cdots \oplus Y_{m} \in Q$) $(\exists X_1 \oplus \cdots \oplus
X_{k} \in P)$
such that $(\tau_{i,l},Y_i)\leq (\sigma_{f(i),l},X_{f(i)})
\wedge (\tau_{i,r},Y_i)\leq (\sigma_{f(i),r},X_{f(i)})$.

We say that
\begin{itemize}
\item$f$ witnesses this extension;
\item part $i$ of
$d$ refines part $f(i)$ of $c$.
\end{itemize}
\end{definition}

\begin{definition}
Part $i$ of condition
 $((\rho_{1,l},\rho_{1,r})
,\ldots,(\rho_{k,l},\rho_{k,r}),P,k)$
 is \emph{acceptable } iff
($\exists X_1
\oplus \cdots \oplus X_{k} \in P$)
 $|X_i|=\infty$
\end{definition}

\subsection{Outline}\label{RT21subsec0}
We will construct a sequence of
forcing conditions $c_0\geq c_1\geq \cdots
\geq c_i\geq \cdots$, together with
a $\mathbf{0}'$-computable function $F:
(c_i, k)\mapsto (side, type)$. The function
$F$ tells how $c_{<e,r>}$ satisfy requirement
$R_e,R_i$, i.e.,
for every forcing condition $c_{<e,r>}$ and part
$k$ of $c_{<e,r>}$, $F(c_{<e,r>},k)=(left, 1)$ iff
$\Phi_e^{\rho^{<e,r>}_{k,l}}(e)\downarrow$;
$F(c_{<e,r>},k)=(left, 0)$ iff
for every $G$ satisfying $c_{<e,r>}$
on part $k$ left side, $\Phi_e^G(e)\uparrow$
(lemma \ref{RT21lem1}).
Similarly for $F(c_{<e,r>},k)=(right, type)$.

The parts of these forcing conditions
form a tree $T$. Nodes on level $s$ of the tree
 represent the parts of condition
$c_s$. Node $j$ is a successor of node $i$ iff
for some $s$, $j$ belongs to level $s+1$,
$i$ belongs to level $s$, and
$f_{s+1}(j)=i$ where $f_{s+1}$
is the witness
of relation $c_{s+1}\leq c_s$
 (see definition \ref{extend}).

We will prove that for any instance
$Y$, there exists a path along the forcing condition
tree $T$, namely part $r_i$ of
condition $c_i$, $i\in\omega$,
such that part $r_i$ of $c_i$ is acceptable
and $\rho^i_{r_i,l}\subseteq A\wedge
\rho^i_{r_i,r}\subseteq\overline{A}$
(subsection \ref{RT21subsec2}).
 And either $\bigcup\limits_{j=1}^\infty
\rho^j_{r_j,l}$ or $\bigcup\limits_{j=1}^\infty
\rho^j_{r_j,r}$ is generalized low
(subsection \ref{RT21subsec3}) and infinite
(lemma \ref{RT21lem3}).

To show the generalized low property,
we prove that either for every
$e$ there exists $r$,
$F(c_{<e,r>},r_{<e,r>}) = (left,\cdot)$;
or for every $r$ there exists $e$,
$F(c_{<e,r>},r_{<e,r>}) = (right,\cdot)$
(lemma \ref{RT21lem4}).
Assume, without loss of generality,
$\forall e\exists r\
F(c_{<e,r>},r_{<e,r>}) = (left,\cdot)$.
Then we prove that given $G_l=\bigcup\limits_{j=1}^\infty
\rho^j_{r_j,l}$ we can $\mathbf{0}'-$compute the
path along $T$, i.e., the function
$i\mapsto r_i$ (section \ref{RT21subsec3}).
Given $e$, to decide
whether $\Phi_e^{G_l}(e)$ halt,
simply find $r$ such that
$F(c_{<e,r>},r_{<e,r>}) = (left,z)$.
Finally, we prove that if
$\forall e\exists r\
F(c_{<e,r>},r_{<e,r>}) = (left,\cdot)$,
then $|G_l|=\infty$ (lemma \ref{RT21lem5}).

\subsection{Constructing
forcing conditions and $F$}\label{RT21subsec1}
We start with condition
$c_{-1}=( (\varepsilon,\varepsilon), \{\omega\},1)$.
Given condition $c_i$, we show how to construct
$c_{i+1}$. Suppose $i+1 = <e,r>$.
We will construct a sequence of forcing conditions
$c_{i,1}\geq c_{i,2}\geq\cdots\geq c_{i,k}=
c_{i+1}$, where $k$ is the number of parts
of $c_i$, each dealing
with a part of $c_i$ to garantee the
successor of that part of $c_{i+1}$ forces
$R_e, R_r$ as described in lemma \ref{RT21lem1}.

Now we construct $c_{i,1}$
Consider the following
$\Pi_1^0$ class,
\begin{align}
[T^{c_i,1}]
=\{X_{1l}&\oplus X_{1r}\oplus X_2\oplus\cdots\oplus
X_k:
(X_{1l}\cup X_{1r})
\oplus X_2\oplus \cdots \oplus X_k\in
P_i, \\ \nonumber
&\hspace{1cm}(\forall Z)
\Phi_e^{(Z\cap X_{1l})/\rho^i_{1,l}}(e)
\uparrow\wedge
\Phi_r^{(Z\cap X_{1r})/\rho^i_{1,r}}(r)
\uparrow
\}
\end{align}
We divide into two cases,
(1) $[T^{c_i,1}]=\emptyset$;
(2) $[T^{c_i,1}]\ne\emptyset$.

If $[T^{c_i,1}]\ne \emptyset$
(in this case we adopt type 0 extension),
split
part 1 of
$c_i$ into two parts with identical
initial segment $(\rho^i_{1,l},\rho^i_{1,r})$,
replace $P_i$ with $[T^{c_i,1}]$,
i.e.,
$$c_{i,1} =
((\rho^i_{1,l},\rho^i_{1,r}),
(\rho^i_{1,l},\rho^i_{1,r}),
(\rho^i_{2,l},\rho^i_{2,r}),
(\rho^i_{3,l},\rho^i_{3,r}),\cdots,
(\rho^i_{k,l},\rho^i_{k,r}),
[T^{c_{i},1}], k+1)$$.
In this case, define
$F(c_{i+1}, 1) =
(left, 0)$,
$F(c_{i+1}, 2) =
(right, 0)$.

Else if
$[T^{c_i,1}]=\emptyset$
(in this case we adopt type 1 extension),
by compactness, there exists
$n$ such that for all
$X= X_1\oplus\cdots\oplus
X_k\in P_i$,
$\rho\in \{0,1\}^n$ there exists
$\tau\in\{0,1\}^n$ such that
either $\tau\subseteq\rho\cap X_1\wedge
\Phi_e^{\rho^i_{1,l}*\tau}(e)\downarrow$ or $
\tau\subseteq \overline{\rho}\cap X_1\wedge
\Phi_r^{\rho^i_{1,r}*\tau}(r)\downarrow$.
Since $P_i\ne \emptyset$,
fix an arbitrary $X\in P_i$.
Suppose $\tau_1,\cdots, \tau_{2^n}$
cover all $\rho\in \{0,1\}^n$ in above way
witnessed by $X=X_1\oplus\cdots\oplus X_k$.
If $\tau_j\subseteq\rho\cap X_1\wedge
\Phi_e^{\rho^i_{1,l}*\tau}(e)\downarrow$
let $\rho^{i,1}_{j,l} =
\rho^i_{1,l}*\tau_j $,
$\rho^{i,1}_{j,r} =
\rho^i_{1,r}*0^n$
and define
$F(c_{i+1},j) =
(left, 1)$;
else if
$
\tau_j\subseteq \overline{\rho}\cap X_1\wedge
\Phi_r^{\rho^i_{1,r}*\tau}(r)\downarrow$
let $\rho^{i,1}_{j,r}
=\rho^i_{1,r}*\tau_j$,
$\rho^{i,1}_{j,l} =
\rho^i_{1,l}*0^n$
and define
$F(c_{i+1},j) =
(right, 1)$.
To construct $c_{i,1}$, split
part 1 of $c_i$ into
$2^n$ many parts, and concatenate
$0^n$ to
initial segments of other parts
of $c_i$, i.e.,
 extend $\rho^i_{j,l},\rho^i_{j,r}$
to $\rho^i_{j,l}*0^n, \rho^i_{j,r}*0^n$
for all $j\neq 1$.
Furthermore, shrink $P_i$
to $Z_1\oplus \cdots \oplus Z_k$ such
that $Z_1\supseteq \bigcup\limits_{j=1}
^{2^n}\tau_j$ and replicate
part $1$  of the shrinked
$P_i$ for $2^n$ many
times, i.e.,
\begin{align}
P_{i,1} = \{
\underbrace{Z_1\oplus\cdots
\oplus Z_1}\limits_{2^n\text{ many}}
\oplus Z_2\oplus Z_3\oplus\cdots
\oplus Z_k:
Z_1\oplus Z_2\oplus\cdots\oplus
Z_k\in P_i\wedge
Z_1\supseteq
\bigcup\limits_{j=1}^{2^n}
\tau_{j}
\}
\end{align}
(Here $\bigcup\limits_{j=1}^{2^n}
\tau_j$ is regarded as a set.)
In summary,

$$c_{i,1} =
((\rho^{i,1}_{1,l},
\rho^{i,1}_{1,r}),\cdots,
(\rho^{i,1}_{2^n,l},
\rho^{i,1}_{2^n,r}),
\cdots,(\rho^{i,1}_{2^n+k-1,l},
\rho^{i,1}_{2^n+k-1,r}),
P_{i,1}, 2^n+k-1)$$
Where for all $j>2^n$,
$\rho^{i,1}_{j,l}
=\rho^{i}_{j-2^n+1,l}*0^n$,
$\rho^{i,1}_{j,r}
=\rho^i_{j-2^n+1,r}*0^n$.
It is clear that
$P_{i,1}\ne\emptyset$
since it at least
contains $X=X_1\oplus\cdots\oplus
X_k$.

\begin{remark}
The purpose to concatenate
$0^n$ to  initial segments of the other
parts is to enable $\mathbf{0}'\oplus G$-
compute the path along the forcing condition
tree, i.e., the function $i\mapsto r_i$
(see lemma \ref{RT21lem3} and subsection \ref{RT21subsec3}).
Note that in this way, all initial segments a forcing condition
are of identical length.
\end{remark}
It is clear that
$c_{i,1}\leq c_i$.
The forcing condition
$c_{i,1}$ dealt with
part $1$ of $c_i$.
After
$c_{i,1}$ is constructed,
construct $c_{i,2}\geq
c_{i,3}\geq \cdots \geq
c_{i,k}$ similarly
to deal with part $2,3,\cdots,k$ of
$c_i$.

The following lemma \ref{RT21lem1}
says that function $F$ tells
how forcing conditions $c_i,i\in\omega$,
satisfy
the requirements.
\begin{lemma}\label{RT21lem1}
The function $F$ is $\mathbf{0}'-$
computable and
\begin{itemize}
\item
If $F(c_{<e,r>}, j) = (left, 1)$,
then for all $G$ satisfying
$c_i$ on part $j$ left side,
i.e., $G\subseteq X_j/\rho^i_{j,l}\wedge
G\supset\rho^i_{j,l}$,
$\Phi_e^G(e)\downarrow$;

Similarly,
\item
If $F(c_{<e,r>}, j) = (left, 0)$,
then for all $G$ satisfying
$c_i$ on part $j$ left side,
$\Phi_e^G(e)\uparrow$;

\item
If $F(c_{<e,r>}, j) = (right, 1)$,
then for all $G$ satisfying
$c_i$ on part $j$ right side,
$\Phi_r^G(r)\downarrow$;
\item
If $F(c_{<e,r>}, j) = (right, 0)$,
then for all $G$ satisfying
$c_i$ on part $j$ right side,
$\Phi_r^G(r)\uparrow$;

\end{itemize}

\end{lemma}

\begin{proof}
The four items are
obvious due to the construction
of $c_{i}$.
Note that the construction
of $c_i$
 and $F$ is uniform in $\mathbf{0}'$.
 Thus, $F\leq_T \mathbf{0}'$.
\end{proof}

To construct the set $G$, we need the following
lemma, which says that the forcing condition
tree is built along all
instances of $\mathsf{RT}_2^1$.
\begin{lemma}\label{RT21lem2}
For any instance of $\mathsf{RT}^1_2$,
$Y$, any forcing condition
$c_i$ there exists
$j\in \omega$,
parts $k_1,\cdots,k_{j}$
of $c_i$
with
$(\forall s\leq j)\rho^i_{k_s,l}
\subseteq Y\wedge
\rho^i_{k_s,r}
\subseteq \overline{Y}$
such that
\[
(\forall X=X_1\oplus \cdots\oplus
X_k\in P_i)\ \
\bigcup\limits_{s=1}^j
X_{k_s}=\omega
\]

\end{lemma}
\begin{proof}
The proof is done by
induction.
Clearly, the lemma holds
for $c_{-1}$.
Assume it holds for
$c_i$.
We show that it holds
for $c_{i,1}$.
Fix an arbitrary instance
of $\mathsf{RT}_2^1$, $Y$.

If $c_{i,1}$ is type 1 extension
of $c_i$, as in section
\ref{RT21subsec1}.
Suppose
for all
$X = X_1\oplus\cdots\oplus X_k
\in P_i$,
$X_{k_1}\cup X_{k_2}\cup\cdots
\cup X_{k_j}=\omega$ and
$\rho^i_{k_s,l}\subseteq
Y\wedge \rho^i_{k_s,r}\subseteq
\overline{Y}$, with
$(\forall s)k_s\neq 1$.
Then it is obvious that
for all $X=X_1\oplus
\cdots\oplus X_{2^n}\oplus \cdots\oplus
X_{2^n+ k-1}\in P_{i,1}$,
$X_{2^n+k_1-1}\cup
X_{2^n+k_2-1}\cup\cdots
\cup X_{2^n+k_j-1}=\omega$,
$\rho^i_{2^n+k_s-1,l}
\subseteq Y\wedge
\rho^i_{2^n+k_s-1,r}
\subseteq \overline{Y}$
for $s=1,2,\cdots,j$ since
there exists $Z_1\oplus Z_2\oplus
\cdots\oplus Z_k\in P_i$, $n\in\omega$
 such that
$Z_{k_s} =X_{2^n+k_s-1}$,
and $\rho^{i,1}_{2^n+k_s-1,l}=
\rho^{i}_{k_s,l}*0^n\wedge
\rho^{i,1}_{2^n+k_s-1,r}
=\rho^i_{k_s,r}*0^n$
for $s=1,2,\cdots,j$.

Suppose
for all
$X = X_1\oplus\cdots\oplus X_k
\in P_i$,
$X_{k_1}\cup X_{k_2}\cup\cdots
\cup X_{k_j}=\omega$ and
$(\forall s\leq j)\rho^i_{k_s,l}\subseteq
Y\wedge \rho^i_{k_s,r}\subseteq
\overline{Y}$ with
$(\forall s>1)k_s\neq 1, k_1=1$.
Take $\rho = Y\upharpoonright_{|\rho^i_{1,l}|+1}^
{|\rho^i_{1,l}|+n}\in \{0,1\}^n$ and
suppose $\tau_h$ covers $\rho$ witnessed by $X$
(recall the construction during type 1 extension),
i.e., $\tau_h\subseteq \rho\cap X_1\vee
\tau_h\subseteq \overline{\rho}\cap X_1$.
Then
 $\rho^{i,1}_{h,l}\subseteq
Y\wedge \rho^{i,1}_{h,r}\subseteq \overline{Y}$
since $\tau_h\subseteq \rho\cap X_1
\rightarrow
\rho^{i,1}_{h,l}
=\rho^i_{1,l}*\tau_h\wedge\rho^{i,1}_{h,r}
=\rho^{i}_{h,r}*0^n $
and
$\tau_h\subseteq \overline{\rho}\cap X_1
\rightarrow
\rho^{i,1}_{h,r}
=\rho^i_{1,r}*\tau_h\wedge\rho^{i,1}_{h,l}
=\rho^{i}_{h,l}*0^n $.
But
for any  $X= X_1\oplus\cdots\oplus
X_{2^n}\oplus\cdots\oplus
X_{2^n+k-1}\in P_{i,1}$
there exists
$Z_1\oplus Z_2\oplus
\cdots\oplus Z_k\in P_i$ with
$Z_1=X_1=X_2=\cdots=
X_{2^n}$,
$Z_{2^n+k_s-1}=X_{k_s}$
for $s=2,3,\cdots,j$. Therefore
$X_h
\cup X_{2^n+k_2-1}\cup
\cdots\cup X_{2^n+k_j-1}=\omega$.
And clearly
$
\rho^{i,1}_{h,l},
\rho^{i,1}_{2^n+k_s-1,l}
\subseteq Y\wedge
\rho^{i,1}_{h,r},
\rho^{i,1}_{2^n+k_s-1,r}
\subseteq \overline{Y}$
for $s=2,3,\cdots,k$.

If $c_{i,1}$ is type 0 extension
of $c_i$. Suppose
for all
$X = X_1\oplus\cdots\oplus X_k
\in P_i$,
$X_{k_1}\cup X_{k_2}\cup\cdots
\cup X_{k_j}=\omega$
and $(\forall s\leq j)
\rho^i_{k_s,l}\subseteq Y\wedge
\rho^i_{k_s,r}\subseteq \overline{Y}$,
$(\forall s)k_s\neq 1$.
Then it is obvious that
for all $X=X_1\oplus
X_2\oplus\cdots\oplus X_{k+1}
\in P_{i,1}$,
$X_{k_1+1}\cup
X_{k_2+1}\cup\cdots
\cup X_{k_j+1}=\omega$,
$(\forall s)\ \rho^{i,1}_{k_s+1,l}
\subseteq Y\wedge
\rho^{i,1}_{k_s+1,r}
\subseteq \overline{Y}$ since
 $Z_1\cup Z_2\oplus
\cdots\oplus Z_k\in P_i$
and $\rho^{i,1}_{k_s+1,l}
=\rho^i_{k_s,l}\wedge
\rho^{i,1}_{k_s+1,r}
=\rho^i_{k_s,r}$
for $s=1,2,\cdots,k$.

 Suppose
for all
$X = X_1\oplus\cdots\oplus X_k
\in P_i$,
$X_{k_1}\cup X_{k_2}\cup\cdots
\cup X_{k_j}=\omega$ and
$(\forall s\leq j)
\rho^i_{k_s,l}\subseteq Y\wedge
\rho^i_{k_s,r}\subseteq \overline{Y}$,
$(\forall s>1)k_s\neq 1,
k_1=1$.
 Let $Z_1\oplus Z_2\oplus\cdots
\oplus Z_{k+1}\in P_{i,1}$ be
arbitrary. But
clearly $\rho^{i,1}_{1,l}
=\rho^{i,1}_{2,l}=\rho^{i}_{1,l}
\subseteq Y\wedge
\rho^{i,1}_{1,r}
=\rho^{i,1}_{2,r}=\rho^{i}_{1,r}
\subseteq \overline{Y}$,
$\rho^{i,1}_{k_s+1,l}=\rho^{i}_{k_s,l}\subseteq
Y\wedge
\rho^{i,1}_{k_s+1,r}=
\rho^{i}_{k_s,r}\subseteq
\overline{Y}$ for
$s=2,3,\cdots,k_j$ and
$Z_1\cup Z_2\cup
Z_{k_2+1}\cup
\cdots
\cup Z_{k_j+1}=\omega$, since
$Z_1\cup Z_2\oplus Z_3\oplus Z_4
\oplus \cdots \oplus Z_{k+1}
\in P_i$.

\end{proof}

Later we need to prove that
given $\bigcup\limits_{i\in\omega}
\rho^i_{r_i,l}$ or
$\bigcup\limits_{i\in\omega}
\rho^i_{r_i,r}$ we can compute the
path through forcing condition
tree $T$, i.e., function
$i\mapsto r_i$
(see subsection \ref{RT21subsec3}).
This needs the following
auxiliary lemma \ref{RT21lem3},
\begin{lemma}\label{RT21lem3}
For any $i,I\in\omega$,
any part $k$ of $c_i$
and any part
$K$ of $c_I$ that is
a decent of part $k$ of
$c_i$, if
$\rho^I_{K,l}-
\rho^i_{k,l}\ne\emptyset$,
then for any part
$k'\ne k$ of $c_i$ and
 any part
$K'$ of $c_I$ that is a decent
of $k'$ of $c_i$, we have,
$\rho^I_{K',l}$ is incomparable
with $\rho^I_{K,l}$.

\end{lemma}
\begin{proof}
This is simply because
the only chance for an initial
segment to add new element is
through type 1 extension.
But whenever some
initial segment is
extended to
$\rho*\tau$ with
$\tau\ne\emptyset$
through type 1 extension
during construction of $c_s$,
$s>i$,
each initial segment of
the other parts of $c_{s-1}$,
 say $\rho'$, is
extended to $\rho'*0^{|\tau|}$.
So $\rho*\tau$ is incomparable with
any $\rho'*0^{|\tau|}$, i.e., any
initial segments whose
 parts in $c_{s-1}$ is different with
 that of $\rho$.

\end{proof}

\subsection{Constructing $G$}\label{RT21subsec2}
Given instance $A$ of
$\mathsf{RT}_2^1$, to construct
 $G\subseteq A\vee G\subseteq
\overline{A}$ that is generalized low,
note that by lemma \ref{RT21lem2},
each forcing condition
$c_i$ admit some part $k$ that is
acceptable and
$\rho^i_{k,l}\subseteq
A\wedge \rho^i_{k,r}\subseteq
\overline{A}$. Also note that
if part $K$ of $c_I$ is acceptable
and part $k$ of $c_i$ is
a parent node of part $K$ of $c_I$,
 then part $k$ of $c_i$ is also
 acceptable.
 Therefore, the acceptable parts
 of all forcing conditions $c_i$,
 $i\in\omega$ form an infinite
 subtree of the whole forcing
 conditions tree $T$. Thus,
 the subtree admit a path, say
 part $r_i$ of $c_i$, $i\in\omega$.
 Consider
 $G_l=\bigcup\limits_{i\in\omega}
 \rho^i_{r_i,l}$,
 $G_r=\bigcup\limits_{i\in\omega}
 \rho^i_{r_i,r}$. It is obvious
 that $G_l\subseteq A\wedge
 G_r\subseteq \overline{A}$.
We will prove that either $G_l$
or $G_r$ is generalized low and infinite.

It is plain to see that,
\begin{lemma}\label{RT21lem4}
Either
$(\forall e)(
\exists r) F(c_{<e,r>},r_{<e,r>})
= (left, \cdot)$
or
$(
\forall r)
(\exists e)
F(c_{<e,r>},r_{<e,r>})
=(right, \cdot)$.

\end{lemma}
We use lemma \ref{RT21lem4}
to prove that at least one of
$G_l,G_r$ is infinite.
\begin{lemma}\label{RT21lem5}
Assume for all $i$
part $r_i$ of $c_i$
is acceptable.

We have, if
$(\forall e)(
\exists r) F(c_{<e,r>},r_{<e,r>})
= (left, \cdot)$ then
$|G_l|=\infty$.

Similarly,
if
$(\forall r)(
\exists e) F(c_{<e,r>},r_{<e,r>})
= (right, \cdot)$ then
$|G_r|=\infty$.

\end{lemma}
\begin{proof}
Assume $(\forall e)(
\exists r) F(c_{<e,r>},r_{<e,r>})
= (left, \cdot)$.
Consider such Turing functional
$E$, $\Phi_E^Z(E)\downarrow$
if and only if $|Z|>E$.
Let $R$ be such that
$F(c_{<E,R>},r_{<E,R>}) =
(left,z)$.
Note that by the construction
of $c_<E,R>$ and $F$,
we have,
either
for all $G$ satisfying
$c_{<E,R>}$ part $r_{<E,R>}$
left side,
$\Phi_E^G(E)\downarrow$;
or
for all $G$ satisfying
$c_{<E,R>}$ part $r_{<E,R>}$,
left side $\Phi_E^G(E)\uparrow$,
depending on $z=1,0$.
But part $r_{<E,R>}$ of
$c_{<E,R>}$ is acceptable.
So there exists
$H$ satisfying $c_{<E,R>}$ on
part $r_{<E,R>}$ left side
that is infinite.
Thus, $\Phi_E^H(E)\downarrow$
by definition of $\Phi_E(E)$.
This implies that for all $G$ satisfying
$c_{<E,R>}$ part $r_{<E,R>}$ left side,
$\Phi_E^G(E)$ halt.
In particular
$\Phi_E^{G_l}(E)$ halt.
This implies $|G_l|>E$.
The proof  is accomplished by noting
that $E$ is arbitrary.

\end{proof}

In the following proof of theorem
\ref{RT21th1} we assume,
without loss of generality,
 $(\forall e)(
\exists r) F(c_{<e,r>},r_{<e,r>})
= (left, \cdot)$. Thus, by lemma
\ref{RT21lem5} $|G_l|=\infty$.

\subsection{Compute $G'$}\label{RT21subsec3}

To compute $G_l'$. We firstly show that
we can compute the function
$i\mapsto r_i $ using
$G_l$ and $\mathbf{0}'$.

Given $i=<e,r>$ to compute $r_{<e,r>}$,
firstly find (uniformly in $i$)
 a Turing
functional $\Phi_E$ such that $E> |\rho^i_{1,l}|$,
 $<E,s> \ > i$ for all $s\in\omega$
and
$\Phi_E^G(E)\downarrow$ iff
$|G|>E$.
By our assumption, $(\forall e)(
\exists r) F(c_{<e,r>},r_{<e,r>})
= (left, \cdot)$, there exists
$R$ such that
$F(c_{<E,R>},r_{<E,R>})=(left,z)$.
Let $R$ be an arbitrary such
integer.

We show that on level $<E,R>$ of
the forcing condition tree $T$,
there exists part $k$ such that
$\rho^{<E,R>}_{k,l}\subset
G_l\wedge
\rho^{<E,R>}_{k,l}-
\rho^i_{f(k,<E,R>,i),l}\ne\emptyset$,
 furthermore for any
part $k'$ of $c_{<E,R>}$ if
 $\rho^{<E,R>}_{k',l}\subset
G_l\wedge
\rho^{<E,R>}_{k',l}-
\rho^i_{f(k',<E,R>, i),l}\ne\emptyset$
then part $k'$ of $c_{<E,R>}$
is a decent
of part $r_i$ of $c_i$
where part $f(k',<E,R>,i)$
of $c_i$ is the accent of part $k'$
of $c_{<E,R>}$.
Thus, to compute $r_i$, simply find
(effectively in $G_l\oplus \mathbf{0}'$)
a number $R$ and part $k$ of $c_{<E,R>}$,
$\rho^{<E,R>}_{k,l}\subset G_l\wedge
\rho^{<E,R>}_{k,l}-\rho^i_{f(k,<E,R>,i),l}\ne\emptyset$.
Then part $r_i$ is simply
$f(k,<E,R>, i)$.

To prove existence of $k$, we show that
$r_{<E,R>}$ is such a part.
As in the proof of lemma \ref{RT21lem5},
it must holds that
$F(c_{<E,R>},r_{<E,R>})=(left,1)$
since part
$r_{<E,R>}$ of
$c_{<E,R>}$ is acceptable.
Therefore $|\{t:
\rho^{<E,R>}_{r_{<E,R>},l}(t)=1\}|\geq E$.
But since $E> |\rho^i_{1,l}|$ therefore
$\rho^{<E,R>}_{r_{<E,R>},l}-\rho^i_{f(r_{<E,R>},
<E,R>,i),l}
\ne\emptyset$.

Now we show that for any
part $k'$ of $c_{<E,R>}$ if
 $\rho^{<E,R>}_{k',l}\subset
G_l\wedge
\rho^{<E,R>}_{k',l}-
\rho^i_{f(k',<E,R>, i),l}\ne\emptyset$
then part $k'$ of $c_{<E,R>}$
is a decent
of part $r_i$ of $c_i$.
Due to lemma
\ref{RT21lem3},
$\rho^{<E,R>}_{k',l}$ is incomparable
with any
$\rho^{<E,R>}_{k,l}$ when
part $k$ of $c_{<E,R>}$ is not a
decent
of part $f(k',<E,R>,i)$ of $c_i$.
But $\rho^{<E,R>}_{k',l}
\subset G_l$ implies
$\rho^{<E,R>}_{k',l}$ equals
to  $\rho^{<E,R>}_{r_{<E,R>},l}$
thus comparable to
$\rho^{<E,R>}_{r_{<E,R>},l}$.
Therefore part $r_{<E,R>}$
of $c_{<E,R>}$ is a decent
of part $f(k',<E,R>,i)$. But
on level $i$ of the forcing
condition tree $T$, part $r_{<E,R>}$
has the unique accent node that is
part $r_i$ of $c_i$.
Therefore $f(k',<E,R>,i)=r_i$.

To compute $G_l'$. Given $e$, to decide
whether $\Phi_e^{G_l}(e)$ halt,
simply compute (in $G_l\oplus \mathbf{0}'$)
 $r$ and
$r_{<e,r>}$ such that
$F(c_{<e,r>},r_{<e,r>})
=(left, z)$. Then
$\Phi_e^{G_l}(e)\downarrow$
iff $z=1$ and
$\Phi_e^{G_l}(e)\uparrow$
iff $z=0$.

\section{Proof of theorem \ref{RT21th2}}
\subsection{Forcing condition}
The forcing condition we use in
the proof of theorem \ref{RT21th2}
is not $\Pi_1^0$ partition class.
The $\Pi_1^0$ partition class
is replaced by a single
low partition, i.e.,
$$((\rho_{1,l},\rho_{1,r}),
(\rho_{2,l},\rho_{2,r}),\cdots,
(\rho_{k,l},\rho_{k,r}),
X_1\oplus X_2 \oplus\cdots\oplus
X_k, k)$$ where
$\bigcup\limits_{j\leq k}
X_j=\omega$ and
$X=X_1\oplus\cdots\oplus X_k$
is low.

Definitions in section \ref{RT21subsec-1} such as
"$c\leq d$", "$G$ satisfy $c$ on part
$j$ left side", "acceptable"
can clearly be inherited.

\subsection{Outline}

We will construct uniformly in $\mathbf{0}'$
 a sequence of
forcing conditions $\cdots c_i\geq c_{i+1}\cdots$
together with a sequence of
$\mathsf{RT}_3^1$ instance initial segment
$\cdots\beta^i\subset \beta^{i+1}\cdots$
such that for any instance of
$\mathsf{RT}_2^1$, $A$, there exists
$G$ encoded by the forcing condition
that does not compute any solution to
$A_3^1=\bigcup\limits_{j=1}^\infty
\beta^j$.

Note that here
by constructing a forcing condition
$c =((\rho_{1,l},\rho_{1,r}),
\cdots, (\rho_{k,l},\rho_{k,r}),
X, k)$ we mean not only to
demonstrate the existence of
$c$ but also compute
the Turing functional $t,t'$ such
that $\Phi_{t}^{\emptyset'}=X$,
$\Phi_{t'}^{\emptyset'}=X'$.
The purpose is to garantee $\mathbf{0}'$-
computability of the function $i\mapsto \beta^i$.

The requirement each forcing condition try
to meet take the form as following.

\begin{emphasise}\label{RT21emph2}
$R_e(\beta):$
$(\forall n\leq |\beta|)
\Phi_e^G(n)\downarrow\Rightarrow
\Phi_e^G(n) = 0$ or
$
\Phi_e^G$ is trivial
or $\Phi_e^G$ violate
$\beta$
deterministically
(i.e., $\Phi_e^G(t)=
\Phi_e^G(s)=1\wedge \beta(s)\ne \beta(t)$ ).
\end{emphasise}
Note that if $G$ satisfy
all $R_e(\beta^e)$ for
a sequence
$\cdots\beta^e\subset \beta^{e+1}\cdots$,
then $G$ fail to compute any non trivial
solution to
$\bigcup\limits_{j=1}^\infty \beta^j$.

We construct the forcing conditions
 and satisfy the requirements in the
following way.
\begin{emphasise}\label{RT21emph1}
For any part $k$ of $c_{<e,r>}$,
either
for every $G $ satisfying $c_{<e,r>}$ on
part $k$ left side, $G$ satisfies
$R_e(\beta^{<e,r>})$ or $R_e(\beta^{<e,r>-1})$;
or
for every $G$ satisfying $c_{<e,r>}$
on part $k$ right side,
$G$ satisfies $R_r(\beta^{<e,r>})$
or $R_r(\beta^{<e,r>-1})$.
\end{emphasise}

Meanwhile, we garantee a lemma \ref{RT21lem2}
holds (see lemma \ref{RT21lem6}).

Once such $c_i, \beta^i, i\in\omega$
are constructed,
given an
instance of $\mathsf{RT}_2^1$, $A$,
there exists an infinite subtree of
the forcing condition tree such that
each node of the subtree represents  an
acceptable part and for every
part $k$ of $c_i$ on that subtree,
$\rho^i_{k,l}\subseteq A\wedge
\rho^i_{k,r}\subseteq \overline{A}$.
Thus, the infinite subtree admit a
path, namely part $r_i$ of $c_i$, $i\in\omega$
such that $\rho^i_{r_i,l}\subseteq
A\wedge \rho^i_{r_i,r}\subseteq\overline{A}$.
We show that either
$G_l=\bigcup\limits_{i\in\omega}
\rho^i_{r_i,l}$ is infinite and
satisfy all $R_e(\beta^e), e\in\omega$; or
$G_r=\bigcup\limits_{i\in\omega}
\rho^i_{r_i,r}$ is infinite
and satisfy all $R_r(\beta^r),r\in\omega$.

\subsection{Constructing the forcing
conditions and $\beta^i,i\in\omega$}

We begin with some definitions
which is also used in \cite{liu2015cone}
\cite{liu2015construct}.
We regard instances of $\mathsf{RT}_3^1$
as functions $\omega\rightarrow \{1,2,3\}$.
\begin{definition}
For an instance of $\mathsf{RT}_3^1$,
$Y$, we say $\Phi^X$ disagree with
$Y$ if and only if there exists
$s,t\in\omega$ $\Phi^X(s)=
\Phi^X(t) = 1$ and
$Y(s)\neq Y(t)$.

\end{definition}

As in the proof of theorem \ref{RT21th1},
let $c_{-1}=( (\varepsilon,\varepsilon), \{\omega\},1)$.
$\beta^{-1}=\varepsilon$.

 Suppose
$i+1=<e,r>$.
Given a condition $c_i=
((\rho^i_{1,l},\rho^i_{1,r})
,(\rho^i_{2,l},\rho^i_{2,r}),
\cdots,(\rho^i_{k_i,l},\rho^i_{k_i,r}),
X^i_1\oplus\cdots\oplus X^i_{k_i}, k_i)$, $\beta^i$,
we construct
a sequence of conditions $c_{i,1}
\geq c_{i,2}\geq \cdots \geq c_{k_i,1}=
c_{i+1}$ together with
$\beta^{i,1}\subset\beta^{i,2}\subset\cdots
\beta^{i,k_i}=\beta^{i+1}$,
 $c_{i,h}$ deals with
part $h$ of $c_{i}$ to garantee
that
 part forces the requirement
$R_e(\beta^{i,h})$ or
$R_e(\beta^{i})$, $R_r(\beta^{i,h})$ 
or $R_r(\beta^{i})$ in the
way mentioned in
\ref{RT21emph1}.
In the following proof, we show how to 
deal with part 1 and construct $c_{i,1}$.

If $(\forall n\leq |\beta^i|)
\Phi_e^{\rho^i_{1,l}}(n)\uparrow\vee
\Phi_e^{\rho^i_{1,l}}(n)\downarrow=0$
or 
$(\forall n\leq |\beta^i|)
\Phi_r^{\rho^i_{1,r}}(n)\uparrow\vee
\Phi_r^{\rho^i_{1,r}}(n)\downarrow=0$
, then we are done by letting $c^{i,1} = c^i$.
Clearly $c^{i,1}$ part $1$ forces
$R_e(\beta^i), R_r(\beta^i)$ as 
\ref{RT21emph1}.

Assume in the following 
that there exists $n,m\leq |\beta^i|$
such that $\Phi_e^{\rho^i_{1,l}}(n)
\downarrow=1\wedge
\Phi_r^{\rho^i_{1,r}}(m)\downarrow=1$.
For an instance of $\mathsf{RT}_3^1$,
$Y$,  let
\begin{align}
[T^{c_i,1}_{Y}]
=\{ X^i_{1l}\oplus X^i_{1r}
&\oplus X^i_2\oplus  \cdots
\oplus X^i_{k_i}:
X^i_{1l}\cup X^i_{1r}=X^i_1,\\ \nonumber
&(\forall Z)
\Phi_e^{(Z\cap X^i_{1l})/
\rho^i_{1,l}}\text{ does not
disagree with }Y\wedge
 \Phi_r^{(Z\cap X^i_{1r})/
\rho^i_{1,r}}\text{ does not
disagree with }Y\}
\end{align}

For an instance of $\mathsf{RT}_3^1$,
$Y$, and $h\in\omega$
 denote by
$Y+h$ the function
$\omega\rightarrow\{1,2,3\}$
$(Y+h)(n) = Y(n)+ (h\mod(3))$.

Consider the $\Pi_1^{0,X^i}$
class
\begin{align}
[T^{c_i,1}] =
\{
Y: [T^{c_i,1}_{Y/\beta^i}],
[T^{c_i,1}_{Y+1/\beta^i}],[T^{c_i,1}_{Y+2
/\beta^i}]
\ne\emptyset
\}
\end{align}

We divide into two cases
(1) $[T^{c_i,1}]\ne\emptyset$;
(2) $[T^{c_i,1}]=\emptyset$.

If $[T^{c_i,1}]\ne\emptyset$, then
by low basis theorem,
there exists  $X^i-$low
instance
of $\mathsf{RT}_3^1$, $ Y$
such that
 $[T^{c_i,1}_{Y/\beta^i}],
[T^{c_i,1}_{Y+1/\beta^i}],[T^{c_i,1}_{Y+2
/\beta^i}]
\ne\emptyset$.
Note that
$[T^{c_i,1}_Z]$ is a
$\Pi_1^{0,X^i\oplus Z}$ class
for any $Z$.
Since $Y$ is $X^i-$low so
$Y+1,Y+2$ are also $X^i-$low.
Therefore again, by low basis theorem,
there exists a $X^i\oplus Y$-low
path through $T^{c_i,1}_{Y+h/\beta^i}$,
for all $h=0,1,2$,
namely,
$X^{i,h}_{1l}\oplus
X^{i,h}_{1r}\oplus
X^i_{2}\oplus\cdots
\oplus X^i_{k_i}$.

 To construct $c_{i,1}$ we apply
 $Cross$ operation to
$X^{i,h}_{1l}\oplus
X^{i,h}_{1r}\oplus
X^i_{2}\oplus\cdots
\oplus X^i_{k_i}$,
$h=0,1,2$ (see also \cite{liu2015cone}),
i.e.,
\begin{align}
X^{i,1} =
(X^{i,0}_{1l}&\cap
X^{i,1}_{1l})\oplus
(X^{i,1}_{1l}\cap
X^{i,2}_{1l})
\oplus
(X^{i,2}_{1l}\cap
X^{i,0}_{1l})
\\ \nonumber
&\oplus
(X^{i,0}_{1r}\cap
X^{i,1}_{1r})
\oplus
(X^{i,1}_{1r}\cap
X^{i,2}_{1r})
\oplus
(X^{i,2}_{1r}\cap
X^{i,0}_{1r})
\oplus X^i_2\oplus
X^i_3\oplus\cdots
\oplus X^i_{k_i}
\end{align}
And replicate
the initial segment $\rho^i_{1,l},
\rho^i_{1,r}$ for 6 times, i.e.,
$$c_{i,1} =
(\underbrace{(\rho^i_{1,l},\rho^i_{1,r}),
\cdots,
(\rho^i_{1,l},\rho^i_{1,r})}\limits_{6\text{ times}},
(\rho^i_{2,l},\rho^i_{2,r}),
(\rho^i_{3,l},\rho^i_{3,r}),
\cdots,
(\rho^i_{k_i,l},\rho^i_{k_i,r}),
X^{i,1}, k_i+5)$$

Clearly $c_{i,1}\leq c_i$.

Note that,
\begin{itemize}
\item
Since $X^{i,h}_{1l}
\oplus X^{i,h}_{1r}
\oplus X^{i}_{2}\oplus \cdots
\oplus X^i_{k_i}$ is $X^i\oplus Y$-low
for $h=0,1,2$ therefore
$X^{i,1}$ is $X^i\oplus Y-$low.
 And because
$Y$ is $X^i-$low,
$X^i$ is low, therefore
$X^{i,1}$ is low. The construction
is clearly uniform, thus
we can $\mathbf{0}'$ compute
(with input $c_i,\beta^i$)
the Turing functional namely
$t_{i,1},t_{i,1}'$, such that
$\Phi_{t_{i,1}}^{\emptyset'}
=X^{i,1}$, $\Phi_{t_{i,1}'}^
{\emptyset'}=(X^{i,1 })'$.

\item
For every $G$ satisfying
$c^{i,1}$  on its first
$3$ parts left side, $\Phi_e^G$
is trivial;
and for every
$G$ satisfying $c^{i,1}$
on its second $3$ parts right side,
$\Phi_r^G$ is trivial.
This is because that
if $G$ satisfy the Mathias condition
$(\rho^i_{1,l}, X^{i,h}_{1l}\cap
X^{i,g}_{1l})$, then
by definition of
$[T^{c_i,1}_{Y+h/\beta^i}],
[T^{c_i,1}_{Y+g/\beta^i}]$,
$\Phi_e^G$ must be a solution of
both $(Y+h)/\beta^i, (Y+g)/\beta^i$.
 However, $(Y+h)/\beta^i, (Y+g)/\beta^i$ share
no common non trivial homogeneous set
containing element
in $\{1,2,\cdots,|\beta^i|\}$ while
every $G$ satisfying the Mathias condition
$(\rho^i_{1,l}, X^{i,h}_{1l}\cap
X^{i,g}_{1l})$, 
$(\exists n\leq |\beta^i|)
\Phi_e^G(n)\downarrow = 1$.

\item
Because each element of $X^i_1$,
$m$, there must exists
$h\ne g\in \{0,1,2\}$ such that
$m$ is contained by either both
$X^{i,h}_{1l},X^{i,g}_{1l}$
or both
$X^{i,h}_{1r},X^{i,g}_{1r}$.
Therefore,
\begin{align}\label{RT21arg1}
(X^{i,0}_{1l}&\cap
X^{i,1}_{1l})
\cup
(X^{i,1}_{1l}\cap
X^{i,2}_{1l})
\cup
(X^{i,2}_{1l}\cap
X^{i,0}_{1l})\\ \nonumber
&\cup
(X^{i,0}_{1r}\cap
X^{i,1}_{1r})
\cup
(X^{i,1}_{1r}\cap
X^{i,2}_{1r})
\cup
(X^{i,2}_{1r}\cap
X^{i,0}_{1r})= X^i_1
\end{align}

\end{itemize}

If $[T^{c_i,1}]=\emptyset$, then
there must exists some
$Y\supset\beta^i$ such that
$[T^{c_i,1}_Y]=\emptyset$.
By compactness, there exists
$n$,
$Y\supset\beta^{i,1}\supset\beta^i$
for any
$\rho\in \{0,1\}^n$
there exists $\tau$ such that
either
$\tau\subseteq \rho\cap X_{1}
\wedge
\Phi_e^{\rho^i_{1,l}*\tau}
$ disagree with $\beta^{i,1}$
or
$\tau\subseteq \overline{\rho}\cap X_1\wedge
\Phi_r^{\rho^i_{1,r}*\tau}
$ disagree with $\beta^{i,1}$.
Suppose $\tau_1,\cdots,\tau_{2^n}$
cover all $\rho\in\{0,1\}^n$ in above
way.
If　$\tau_j\subseteq
\rho\cap X_1\wedge
\Phi_e^{\rho^i_{1,l}*\tau}$
disagree with $\beta^{i,1}$, then
let $\rho^{i,1}_{j,l} =
\rho^i_{1,l}*\tau_j,
\rho^{i,1}_{j,r} = \rho^i_{1,r}$;
else if
$\tau\subseteq \overline{\rho}\cap X_1\wedge
\Phi_r^{\rho^i_{1,r}*\tau}
$ disagree with $\beta^{i,1}$,
then let
$\rho^{i,1}_{j,r} =
\rho^i_{1,r}*\tau_j,
\rho^{i,1}_{j,l} = \rho^i_{1,l}$.
To construct $c_{i,1}$,
split the initial segment of
part $1$ of $c_i$
into $2^n$ many initial segments
as above and preserve all initial
segments of other parts of $c_i$.
Furthermore,  replicate
part $1$  of
$X^i$ for $2^n$ many
times, i.e.,
\begin{align}
X^{i,1}=
\underbrace{X_1\oplus\cdots
\oplus X_1}\limits_{2^n\text{ many}}
\oplus X_2\oplus X_3\oplus\cdots
\oplus X_k
\end{align}

In summary,

$$c_{i,1} =
((\rho^{i,1}_{1,l},
\rho^{i,1}_{1,r}),\cdots,
(\rho^{i,1}_{2^n,l},
\rho^{i,1}_{2^n,r}),
\cdots,(\rho^{i,1}_{2^n+k-1,l},
\rho^{i,1}_{2^n+k-1,r}),
X^{i,1}, 2^n+k-1)$$
Where for all $j>2^n$,
$\rho^{i,1}_{j,l}
=\rho^{i}_{j-2^n+1,l}$,
$\rho^{i,1}_{j,r}
=\rho^i_{j-2^n+1,r}$.
It is clear that
$c_{i,1}\leq c_i$
and $X^{i,1}$ is low.

The forcing condition
$c_{i,1}$ dealt with
part $1$ of $c_i$.
After
$c_{i,1}$ is constructed,
construct $c_{i,2}\geq
c_{i,3}\geq \cdots \geq
c_{i,k}$ similarly
to deal with part $2,3,\cdots,k$ of
$c_i$.

Similar to lemma \ref{RT21lem1},
 we can show that $c_{i+1}$ satisfies
 the requirements $R_e(\beta^{i+1})$
 or $R_e(\beta^{i})$,
 $R_r(\beta^{i+1})$ or $R_r(\beta^i)$
  as in \ref{RT21emph1}.
\begin{lemma}\label{RT21lem10}
For every $e,r$, for every
part $k$ of $c_{<e,r>}$,

either
for every
$G$ satisfying $c_{<e,r>}$ on
part $k$ left side,
$G$ satisfy $R_e(\beta^{<e,r>})$
or $R_e(\beta^{<e,r>-1})$;

or
for every
$G$ satisfying $c_{<e,r>}$ on
part $k$ right side,
$G$ satisfy $R_r(\beta^{<e,r>})$
or $R_r(\beta^{<e,r>-1})$.

\end{lemma}
In the first case we say that
$c_{i+1}$ part $k$ progresses on the left
side and
 in the second case we say that
 $c_{i+1}$ part $k$ progresses on the right
 side.

To construct $G$,
we establish the following lemma
that is exactly the same as
lemma \ref{RT21lem2}.
\begin{lemma}\label{RT21lem6}
For any instance of $\mathsf{RT}^1_2$,
$Y$, any
$c_i$, there exists
$j\in \omega$,
parts $k_1,\cdots,k_{j}$
of $c_i$
with
$\rho^i_{k_s,l}
\subseteq Y\wedge
\rho^i_{k_s,r}
\subseteq \overline{Y}$
such that
\[
\bigcup\limits_{s=1}^j
X^i_{k_s}=\omega
\]

\end{lemma}
\begin{proof}
The proof concern \ref{RT21arg1}
and proceeds
exactly the same as \ref{RT21lem2}.
\end{proof}

\begin{remark}
In case
$[T^{c_i,1}]=\emptyset$,
differently
with the proof of theorem \ref{RT21th1},
we need not concatenate $0^n$ to
initial segments of other parts. Because
given $G$ we need not compute the
path through the forcing condition tree,
i.e., we do not need lemma \ref{RT21lem3}
here. But we need to prove additionally
that $A_3^1 = \bigcup\limits_{j=1}^\infty
\beta^j$ is $\mathbf{0}'$-computable
as in the following lemma \ref{RT21lem9}.

\end{remark}

\begin{lemma}\label{RT21lem9}
The $\mathsf{RT}_3^1$
instance $A_3^1 = \bigcup\limits_{j=1}^\infty
\beta^j$ is $\mathbf{0}'$-computable.

\end{lemma}
\begin{proof}
The proof is accomplished by noting
that the construction is effective in
$\mathbf{0}'$.
\end{proof}

\subsection{Constructing  $G$}
Let
$A_3^1 = \bigcup\limits_{j=1}^\infty
\beta^j$.
Given instance $A$ of
$\mathsf{RT}_2^1$, to construct
 $G\subseteq A\vee G\subseteq
\overline{A}$ that does not compute
any non trivial solution to $A_3^1$,
note that by lemma \ref{RT21lem2},
each forcing condition
$c_i$ admit some part $k$ that is
acceptable and
$\rho^i_{k,l}\subseteq
A\wedge \rho^i_{k,r}\subseteq
\overline{A}$. Also note that
if part $K$ of $c_I$ is acceptable
and part $k$ of $c_i$ is
an accent node of part $K$ of $c_I$,
 then part $k$ of $c_i$ is also
 acceptable.
 Therefore, the acceptable parts
 of all forcing conditions $c_i$,
 $i\in\omega$ form an infinite
 subtree of the whole forcing
 conditions tree $T$. Thus,
 the subtree admit a path, say
 part $r_i$ of $c_i$, $i\in\omega$.
 Consider
 $G_l=\bigcup\limits_{i\in\omega}
 \rho^i_{r_i,l}$,
 $G_r=\bigcup\limits_{i\in\omega}
 \rho^i_{r_i,r}$. It is obvious
 that $G_l\subseteq A\wedge
 G_r\subseteq \overline{A}$.
We will prove that either $G_l$
or $G_r$ fails to compute
any non trivial solution of $A_3^1$ and is infinite.

It is plain to see that,
\begin{lemma}\label{RT21lem7}
Either
$(\forall e)(
\exists r)$ $c_{<e,r>}$ part
$r_{<e,r>}$ progresses on left side
or
$(\forall r)(
\exists e)$ $c_{<e,r>}$ part
$r_{<e,r>}$ progresses on right side.

\end{lemma}
(Recall the definition of progress on the
paragraph after lemma \ref{RT21lem10})

We use lemma \ref{RT21lem7}
to prove that at least one of
$G_l,G_r$ is infinite.
\begin{lemma}\label{RT21lem8}
Assume for all $i$
part $r_i$ of $c_i$
is acceptable.

We have, if
$(\forall e)(
\exists r)$ $c_{<e,r>}$ part
$r_{<e,r>}$ progresses on left side,
 then
$|G_l|=\infty$.

Similarly,
if
$(\forall r)(
\exists e)$ $c_{<e,r>}$ part
$r_{<e,r>}$ progresses on right side,
 then
$|G_r|=\infty$.

\end{lemma}
\begin{proof}
The proof goes almost the same
as \ref{RT21lem5}. 
Without loss of generality
suppose $(\forall e)(
\exists r)$ $c_{<e,r>}$ part
$r_{<e,r>}$ progresses on left side.
For any $c$ it suffices to prove that 
$|G_l|>c$.
Consider such Turing functional
$e$ that for any $X$ 
$\Phi_e^X(1)= 1$ and 
$(\forall n\ne 1)\Phi_e^X(n)\downarrow=1$ iff
$|X|>c$. Since 
there exists $r$ such that
$c_{<e,r>}$ part $r_{<e,r>}$ 
progress on the left
while $\Phi_e^X(1)= 1$,
therefore it must holds that 
either for any $G$ satisfying
$c_{<e,r>}$ part
$r_{<e,r>}$ left side $\Phi_e^G$
violate $\beta^{<e,r>}$ explicitly
or $\Phi_e^G$ is trivial.
But $c_{<e,r>}$ part $r_{<e,r>}$
is acceptable thus by definition
of $\Phi_e$, there exists
$G$ satisfying
$c_{<e,r>}$ part $r_{<e,r>}$ left side
such that
$\Phi_e^G$ is non trivial.
Therefore $\Phi_e^G$
violate $\beta^{<e,r>}$ explicitly
for all $G$ satisfying
$c_{<e,r>}$ part $r_{<e,r>}$ left side.
This implies that there exists
$n\ne 1$ such that $\Phi_e^{\rho_{r_{<e,r>,l}}}
(n)=1$ which implies $|\rho_{r_{<e,r>,l}}|>c$.

\end{proof}
Lemma \ref{RT21lem7}, \ref{RT21lem8}
together proved that either
$G_l\subseteq A$ is infinite and
does not compute any non trivial
solution of $A_3^1$
or
$G_r\subseteq \overline{A}$ is infinite and
does not compute any non trivial
solution of $A_3^1$.
To see this, suppose
$(\forall e)(
\exists r)$ $c_{<e,r>}$ part
$r_{<e,r>}$ progresses on left side,
we show that for any $e$ 
$\Phi_e^{G_l}$ violate
$A_3^1$ explicitly of is trivial.
Given $e$ there exists infinitely many
Turing functional $\Phi_{e_j},j\in\omega$
such that all $\Phi_{e_j}, j\in\omega$
are exactly the same as $\Phi_e$.
However, for every $e_j$
there exists
$r_j$  such that $G_l$ satisfy 
$R_{e_j}(\beta^{<e_j,r_j>})$
or $R_{e_j}(\beta^{<e_j,r_j>-1})$.
Therefore by definition of requirement
\ref{RT21emph2}
(also note that $\bigcup\limits_{j\in\omega}
\beta^{<e_j,r_j>-1}=A_3^1$)
 either
$(\forall n\in\omega)\Phi_e^{G_l}(n)\downarrow\
\Rightarrow \Phi_e^{G_l}(n)=0$ or 
$\Phi_e^{G_l}$ violate some $\beta^{<e_j,r_j>}$
explicitly or $\Phi_e^{G_l}$ is trivial.

Thus the proof of theorem
\ref{RT21th2} is accomplished.

\section{Further discussion and questions}
\label{RT21secfurther}
The results are of technical interest.
The proof is different with that of
\cite{liu2015cone}
or \cite{Dzhafarov2009} in the sense that here we
construct forcing conditions
along all instances of the problem ($\mathsf{RT}_2^1$).
Where in \cite{liu2015cone} we construct the objective
set along a single instance. In another words,
we pre choose a path through the forcing condition
tree during the construction, and need not look
at the construction else where in
\cite{liu2015cone}. But here, to construct
a set of "homogeneous"
solutions intersecting with
all instances (of $\mathsf{RT}_2^1$),
it is necessary to construct the
forcing conditions along all instances simultaneously.
The difference is reflected by the type
0 extension and lemma \ref{RT21lem3},
\ref{RT21lem6}.

The results in this paper and
many results arising recently \cite{lerman2013separating}
\cite{patey2015iterative}
\cite{wangdefinability},
which says there exists somewhat weak
solution to some instance,
motivate us to find pure combinatorial
conditions for a problem to
admit "weak" solution in all of its
instance. For example, taking
"weak" to be generalized low.

The proof that $\mathsf{RT}_2^1$ admit
generalized low solution in every
instance is somewhat robust. Therefore,
we wonder if many other problems also
has this property.

\begin{question}
Is there a purely combinatorial
condition  that is
necessary and sufficient
for a problem $P$ to admit
generalized low solution in
all its instances?

\end{question}

For many known 
problems, say $\mathsf{(S)CAC}$,
$\mathsf{(S)ADS}$,
$\mathsf{(S)RT}_2^2$ etc, they either
encode large functions
(every function $g$ there exists
instance of the problem such that 
all non trivial solutions compute
some $f\geq g$)  thus
encode hyperarithmetic degree or 
admit generalized low solution.
It'd be interesting to see
some counterexamples.
\begin{question}
Does there exists a "natural"
problem that neither 
encode large function
nor admit generalized low
solutions?
\end{question}

\bibliographystyle{amsplain}
\bibliography{simp}

\providecommand{\bysame}{\leavevmode\hbox to3em{\hrulefill}\thinspace}
\providecommand{\MR}{\relax\ifhmode\unskip\space\fi MR }
\providecommand{\MRhref}[2]{%
  \href{http://www.ams.org/mathscinet-getitem?mr=#1}{#2}
}
\providecommand{\href}[2]{#2}
\begin{thebibliography}{10}

\bibitem{Cholak2001}
P.A. Cholak, C.G. Jockusch, and T.A. Slaman, \emph{{On the strength of Ramsey's
  theorem for pairs}}, Journal of Symbolic Logic \textbf{66} (2001), no.~1,
  1--55.

\bibitem{chong2014metamathematics}
Chitat Chong, Theodore Slaman, and Yue Yang, \emph{The metamathematics of
  stable ramsey¡¯s theorem for pairs}, Journal of the American Mathematical
  Society \textbf{27} (2014), no.~3, 863--892.

\bibitem{dzhafarov2015cohesive}
Damir Dzhafarov, \emph{Cohesive avoidance and strong reductions}, Proceedings
  of the American Mathematical Society \textbf{143} (2015), no.~2, 869--876.

\bibitem{Dzhafarov2009}
Damir~D. Dzhafarov and Carl~G. Jockusch, \emph{Ramsey¡¯s theorem and cone
  avoidance}, Journal of Symbolic Logic \textbf{74} (2009), 557--578.

\bibitem{hirschfeldt2014slicing}
Denis~R Hirschfeldt, \emph{Slicing the truth: On the computable and reverse
  mathematics of combinatorial principles}, World Scientific, 2014.

\bibitem{hirschfeldt2015notions}
Denis~R Hirschfeldt and Carl~G Jockusch~Jr, \emph{On notions of computability
  theoretic reduction between $\pi$1 2 principles}, To appear (2015).

\bibitem{jockusch1972ramsey}
Carl~G Jockusch, \emph{Ramsey's theorem and recursion theory}, The Journal of
  Symbolic Logic \textbf{37} (1972), no.~02, 268--280.

\bibitem{lerman2013separating}
Manuel Lerman, Reed Solomon, and Henry Towsner, \emph{Separating principles
  below ramsey's theorem for pairs}, Journal of Mathematical Logic \textbf{13}
  (2013), no.~02, 1350007.

\bibitem{liu2012rt}
Jiayi Liu, \emph{$\mathsf{RT_2^2}$ does not imply $\mathsf{WKL_0}$}, Journal of
  Symbolic Logic \textbf{77} (2012), no.~2, 609--620.

\bibitem{liu2015cone}
Lu~Liu, \emph{Cone avoiding closed sets}, Transactions of the American
  Mathematical Society \textbf{367} (2015), no.~3, 1609--1630.

\bibitem{patey2015iterative}
Ludovic Patey, \emph{Iterative forcing and hyperimmunity in reverse
  mathematics}, arXiv preprint arXiv:1501.07709 (2015).

\bibitem{patey2015open}
\bysame, \emph{Open questions about ramsey-type statements in reverse
  mathematics}, arXiv preprint arXiv:1506.04780 (2015).

\bibitem{patey2015strength}
\bysame, \emph{The strength of the tree theorem for pairs in reverse
  mathematics}, arXiv preprint arXiv:1505.01057 (2015).

\bibitem{patey2015weakness}
\bysame, \emph{The weakness of being cohesive, thin or free in reverse
  mathematics}, arXiv preprint arXiv:1502.03709 (2015).

\bibitem{patey2016proof}
Ludovic Patey and Keita Yokoyama, \emph{The proof-theoretic strength of
  ramsey's theorem for pairs and two colors}, arXiv preprint arXiv:1601.00050
  (2016).

\bibitem{ramsey1930problem}
F.P. Ramsey, \emph{{On a problem of formal logic}}, Proceedings of the London
  Mathematical Society \textbf{2} (1930), no.~1, 264.

\bibitem{seetapun1995strength}
David Seetapun, Theodore~A Slaman, et~al., \emph{On the strength of ramsey's
  theorem}, Notre Dame Journal of Formal Logic \textbf{36} (1995), no.~4,
  570--582.

\bibitem{simpson1999subsystems}
S.G. Simpson, \emph{{Subsystems of second order arithmetic}}, Springer, 1998.

\bibitem{wangdefinability}
Wei Wang, \emph{The definability strength of combinatorial principles, 2014. to
  appear}.

\end{thebibliography}

\end{document}